\documentclass[12pt]{article}
\usepackage{amsmath}
\usepackage{amsthm}
\usepackage{amsfonts}
\usepackage{amssymb}
\usepackage{verbatim}
\usepackage{amscd}
\usepackage{enumerate}
\usepackage[pdftex]{graphicx}
\usepackage{wrapfig}

\newtheoremstyle{thm}{1.5ex}{1.5ex}{\itshape\rmfamily}{}
{\bfseries\rmfamily}{}{2ex}{}

\newtheoremstyle{rem}{1.3ex}{1.3ex}{\rmfamily}{} 
{\itshape}
{} {1.5ex}{}

\newtheoremstyle{theorem}{1.5ex}{1.5ex}{\itshape\rmfamily}{} {\bfseries\rmfamily}{}{2ex}{}

\newtheoremstyle{def}{1.5ex}{1.5ex}{\slshape\rmfamily}{} {\bfseries\rmfamily}{}{2ex}{}

\newtheoremstyle{rem}{1.3ex}{1.3ex}{\rmfamily}{} {\itshape}
{} {1.5ex}{}

\theoremstyle{theorem}
\newtheorem{theorem}{Theorem}[section]
\newtheorem{lemma}[theorem]{Lemma}

\newtheorem*{Main Theorem}{Main Theorem.}
\newtheorem{corollary}[theorem]{Corollary}

\theoremstyle{def}

\theoremstyle{rem}
\newtheorem{remark}{{\itshape Remark}}[]

\numberwithin{equation}{section}

\begin{document}

\title
{\Large{Large Scale Properties of the IIIC for 2D Percolation}}

\author
{\large L. Chayes$^{1}$, P. Nolin$^{2,3}$}
\date{}
\maketitle

\vspace{-4mm}

\centerline{${}^1$\textit{Department of Mathematics, UCLA, Los Angeles,
CA 90059--1555 USA }}
\centerline{${}^2$\textit{D\'epartement de Math\'ematiques et Applications, ENS, 75230 Paris cedex 05, France }}
\centerline{${}^3$\textit{Laboratoire de Math\'ematiques, Universit\'e Paris--Sud, 91405 Orsay cedex, France }}

\begin{quote}
{\footnotesize
{\bf Abstract:}
We reinvestigate the 2D problem of the {\it inhomogeneous incipient infinite cluster} where, in an independent percolation model, the density decays to $p_c$ with an inverse power, $\lambda$, of the distance to the origin.  Assuming the existence of critical exponents (as is known in the case of the triangular site lattice) if the power is less than $1/\nu$, with $\nu$ the correlation length exponent, we demonstrate an infinite cluster with scale dimension given by $D_H = 2 - \beta\lambda$.  Further, we investigate the critical case $\lambda_c = 1/\nu$ and show that {\it iterated} logarithmic corrections will tip the balance between the possibility and impossibility of an infinite cluster.
}
\end{quote}

\section{Introduction}
A while ago, one of us -- in collaboration with others -- introduced a notion of
{\it inhomogeneous percolation}  \cite{CCD} that was demonstrated to have some interesting properties.  The model is defined by allowing the density parameter to vary, e.g.~with the distance to the origin, in such a way that the system will just barely house an infinite cluster.  Explicitly, one looks at
\begin{equation}
p(r) \cong  p_c + \frac{1}{r^{\lambda} }
\end{equation}
where $r$ denotes distance to the origin (and it should be assumed $r$ is large enough so that the right--hand side makes sense).  
For $d = 2$, under the assumption of the existence of critical exponents, it was found that if $\lambda < \lambda_c = 1/\nu$ the origin belongs to an infinite cluster with positive probability, while this probability vanishes if $\lambda > \lambda_{c}$.  In the preceding, $\nu$ is the correlation length exponent -- precise definitions later -- and, in fact, an equivalent but more awkward statement can be made without reference to exponents.  For $\lambda < \lambda_c$, we will refer to the infinite object as the {\it inhomogeneous incipient infinite cluster} (IIIC).  

In the ensuing time, there have been two landmarks in 2D percolation, namely the works of Kesten in the late 1980's (\cite{Ke2}, \cite{Ke3}, \cite{Ke4}) wherein critical scaling relations were established modulo the existence of certain critical exponents, and the more recent works by (various combinations of) Lawler, Schramm, Smirnov and Werner (\cite{S}, \cite{LSW4}, \cite{SmW}) where the existence of these exponents -- and their values -- was established for the case of the triangular site lattice using the connection, in the scaling limit, to the Schramm-Loewner Evolution (SLE) with parameter 6. Thus, most of the original results can be sharpened at least in certain cases. However, such matters are largely automatic.  

The main result of this note concerns the large scale structure of the percolating cluster. In particular, it turns out that these objects have a well--defined Hausdorff dimension (more precisely scale--dimension) that is given by 
\begin{equation}
D_H  =  2 - \beta\lambda 
\end{equation}
for $0 < \lambda < \lambda_c$, where $\beta$ is the percolation density, or order parameter exponent.  It is noted that as $\lambda \downarrow \lambda_c$ this dimension matches that of the standard IIC as discussed in e.g.~\cite{StanleyI}, \cite{StanleyII} and proved, modulo the existence of exponents, in \cite{Ke2}.   
Further, we discuss the borderline case, informally $p(r) - p_c = r^{-1/\nu}K(r)$ where $K(r)$ is a ``correction''.  It turns out that at the border, the balance is {\it very} delicate and 
\begin{equation}
\label{DD}
K(r) \sim [\log \log r]^{1/\nu} 
\end{equation}
 will determine the presence or absence of infinite structures.  
 All results save the latter can be stated without apology for the triangular site model; a statement along the lines of Eq.(\ref{DD}) requires {\it strong} existence of power laws which, at this time, has not been established, and we will be content with a statement that circumvents this necessity. 
\section{Setup and Statement of Theorems}

\subsection{Setting}

We consider any of the standard 2D percolation models -- explicitly any model for which the results of \cite{Ke2} -- \cite{Ke4} can be established.  
In particular, what is needed is reflection symmetry about one of the coordinate axes and overall rotational invariance by any angle in $(0,\pi)$.  However, it is sufficient that the reader keeps in mind only the bond or site problems on the square or triangular lattice (unfortunately, the latter requires the use of parallelograms rather than rectangles and, since the triangular site model is where the strongest results are known, we are forced to carry this terminology).

For the purposes of this note, it is assumed that the reader is familiar with the standard fare associated with these sorts of percolation problems; additional background material can be found in the reference \cite{Gr}.

Let us now fix some working notation/definitions: we take the vertical axis to be the axis of reflection symmetry and $r(z)= \|z\|_{\infty}$ (abv.~$\|z\|$ since, in any case, all norms are equivalent) will denote the infinity norm of a site $z$ as measured with respect to the $x$--axis and the axis related to this by the angle of rotation symmetry. The set of points at distance at most $N$ from a site $z$ is a rhombus centered at this site and whose sides line up with the above mentioned axes. It will be denoted by $S_N(z)$, its boundary being the set $\partial S_N(z)$ of points at distance exactly $N$ from $z$. We will refer to $S_N(0)$ simply as $S_N$. We will often use the fact that
\begin{equation}
|S_N(z)| \leq C_0 N^2
\end{equation}
for some constant $C_0$ that may depend on the lattice.

Bonds or sites (as appropriate) will be {\it occupied} with probability $p$ and {\it vacant} with probability $1-p$, independently of each other. We denote by $P_\infty (p)$ the probability that the site at the origin is connected to infinity, and by $p_c \in (0,1)$ the percolation threshold:  $P_\infty (p) > 0$ {\it iff} $p>p_c$. 
If $A$ and $B$ are sets (which, for convenience, will include the case ``infinity''), then we use the notation 
$A \leadsto B$
to denote the event that some site in $A$ is connected to some site in $B$.  If the connection is required to take place using exclusively the sites of some other set $C$, we write
$A \negthinspace \mathop{_{\ ^{\leadsto}\ }^{~C}} \negthinspace B$.  
Finally, all quantities adorned by an $\ast$ will pertain to the {\it dual model}.

We will make use of the following one--arm probability
\begin{equation}
\pi(N) := \mathbb{P}_{p_c}(\{0 \leadsto \partial S_N\})
\end{equation}
and, in addition,
\begin{equation}
\pi(n | N) := \mathbb{P}_{p_c}(\{\partial S_n \leadsto \partial S_N\}).
\end{equation}
The so-called Russo-Seymour-Welsh theory (see e.g. \cite {Gr}) implies that
\begin{equation}
\label{E}
\pi(n|2N), \pi(\lfloor n/2 \rfloor|N) \geq D_1 \pi(n|N)
\end{equation}
and
\begin{equation}
\label{alpha}
D_2  \Bigg[ \frac{n}{N }\Bigg]^{\mu} \leq \pi(n | N) \leq D_3 \Bigg[ \frac{n}{N }\Bigg]^{\mu'}
\end{equation}
for some constants $0 < D_1, D_2, D_3, \mu, \mu' < \infty$. We will later have use for $\mu < 2$ so we may as well take $\mu = \frac12$ (this may be derived by a variant of the ``example'' (3.15) in \cite{KvdB} where one now uses blocks of size $n$ instead of individual sites to obtain that $\frac Nn\pi^2(n|N)$ is bounded below by a constant).
Finally, we also have
\begin{equation}
\label{F}
D_4 \pi(n_0 | n_2) \leq \pi(n_0 | n_1) \times \pi(n_1 | n_2) \leq D_5 \pi(n_0 | n_2)
\end{equation}
whenever $n_0 < n_1 < n_2$, for some $0 < D_4, D_5 < \infty$.

\subsection{Correlation Lengths}

We will assume throughout that $p > p_c$, as this is the only case we are interested in. The primary correlation length used in this note, describing connection probabilities, will be defined via the dual model: let $z^\ast$ denote a site on the dual lattice and let $\tau_{0^\ast\negthinspace,z^\ast}^\ast (p)$ denote the probability of a dual connection between the dual origin and $z^\ast$, i.e.~the event $\{ 0^\ast \negthinspace \mathop{_{\ ^{\negthinspace\leadsto}\ }^{~\ast}} \negthinspace z^\ast \}$.
Finally, let $\tau^{\ast}_n(p)$ denote the maximum of such connection probabilities with $\|z^\ast\|$ ($= \|z^\ast\|_\infty$) within a lattice spacing of $n$.  Then, the correlation length $\xi(p)$ is defined by
\begin{equation}
 \label{}
\lim_{n\to\infty}[\tau^{\ast}_n(p)]^{\frac{1}{n}} = e^{-\frac{1}{\xi(p)}}
\end{equation}
with $\xi = 0$ if $p = 1$. As is well known, the function $\xi$ is continuous, monotone and divergent at $p = p_c$; the power of $p - p_c$ with which this function purportedly diverges ``defines'' the exponent $\nu$.
Further, for future reference, the functions 
$\tau_{0^\ast\negthinspace,z^\ast}^\ast$ obey the {\it a priori} bounds 
\begin{equation}
\label{apb}
\tau_{0^\ast\negthinspace,z^\ast}^\ast (p) \leq  
e^{-\frac{\|z^{\ast}\|}{\xi(p)}}.
\end{equation}

Another frequently used correlation length is the (quadratic) mean radius $\tilde{\xi}(p)$ of a finite cluster, defined by
\begin{equation}
\tilde{\xi}(p) = \Bigg[ \frac{1}{\mathbb{E}_p \big[|C(0)| ; |C(0)| < \infty \big]} \sum_z \|z\|^2 \mathbb{P}_p \big(\text{$\{0 \leadsto z\}$ and $|C(0)| < \infty$}\big) \Bigg]^{1/2}.
\end{equation}
We shall also have use for an auxiliary correlation length -- often called finite--size correlation length -- which we will denote by $L(p)$; technically this depends on an additional parameter $\delta$ which will be notationally suppressed.  In this note, the length $L(p)$ will be defined as the smallest $3\times1$ parallelogram -- with the short angle being the angle of the rotation symmetry -- such that the probability of an occupied crossing exceeds $1 - c\delta$.  Here $c$ is a particular constant of order unity and $\delta$ may be chosen arbitrarily in $(0,1)
$\footnote{Kesten proved in fact the following in \cite{Ke4}: for any (fixed) $\delta_1, \delta_2 \in (0,1)$, we have $L(p,\delta_1) \asymp L(p,\delta_2)$.}.
The key item is that if the above mentioned  estimate on the crossing probability is satisfied then, upon tripling the length scale, the improved estimate becomes $1-c\delta^{3}$, so that on further rescalings, crossing probabilities tend to one exponentially fast.  In particular, for all $n$, all $p > p_c$, the probability of a dual crossing of an $n\times 3n$ parallelogram is bounded above by a constant times $e^{-n/L(p)}$, which implies that
\begin{equation}
\label{theta}
P_{\infty}(p) \geq c_0 \: \mathbb P_{p}(\{ 0 \leadsto \partial S_{L(p)} \})
\end{equation}
for some universal constant $0 < c_0 < \infty$.

It is noted that for length scales smaller than $L(p)$, crossing probabilities of these shorter and longer parallelograms are bounded above and below by strictly positive constants that depend only on the aspect ratio, as this is the situation at $p = p_c$ on {\it all} length scales.  This is proved by a variant of the Russo--Seymour--Welsh theorem, see e.g.~the relevant lemmas in \cite{Ke0} Ch.~6.  Obviously, the same kind of bounds hold for dual crossings.

It was shown in \cite{Ke4} that $L(p)$ and $\tilde{\xi}(p)$ are uniformly bounded above and below by ($\delta$ dependent) multiples of one another, that is, in the notation of Kesten,
$$
L(p) \asymp \tilde{\xi}(p).
$$
It was mentioned in \cite{BCKS} that the relation $\xi(p) \asymp L(p)$ was known; however to the authors' knowledge, there is no published proof.  In any case, at least for 2D percolation problems, it is not hard to show it -- we will provide the details in a short appendix -- thus all these correlation lengths are equivalent. To define the model we have a slight preference for $\xi$, which is continuous and monotone, but for proofs the length $L$ will most often be more practical.

Finally, as alluded to above, concerning asymptotic issues, we will use Kesten's notations: for two positive functions $f$ and $g$, $f \asymp g$ means that there exist two positive and finite constants $C_a$ and $C_b$ such that $C_a g \leq f \leq C_b g$ (so that their ratio is bounded away from $0$ and $+ \infty$), whereas $f \approx g$ means that ${\log f} / {\log g} \to 1$ (``logarithmic equivalence'').  These items will refer to $p \to p_c$ or $N \to \infty$ depending on the context.

Kesten proved in \cite{Ke4} than the one--arm probability stays of the same order of magnitude if we do not go beyond the characteristic scale: more precisely,
\begin{equation}
\label{onearm}
\mathbb{P}_p(0 \leadsto \partial S_n) \asymp \mathbb{P}_{p_c}(0 \leadsto \partial S_n)
\end{equation}
uniformly in $p$ and $n \leq L(p)$. In particular, we can combine it with Eq.(\ref{theta}):
\begin{equation}
\label{KClassic}
P_{\infty}(p)  \asymp \mathbb P_{p}(\{ 0 \leadsto \partial S_{L(p)} \})
\asymp \mathbb P_{p_c}(\{ 0 \leadsto \partial S_{L(p)} \}).
\end{equation}
This result (stated also in \cite{Ke4}) will prove to be very useful when dealing with small boxes on which the parameter does not vary too much.

\subsection{Description of the Model}

We let $\alpha: [0,+\infty) \to (0,1-p_c]$ denote the inverse function of $\xi$ with argument of the increment above threshold:
\begin{equation}
\label{r}
\xi(p_c + \alpha(r))  =  r.
\end{equation}
Letting $w \in (0,1)$, our inhomogeneous density will be defined by
\begin{equation}
\label{p(r)}
p(\negthinspace z) := p_c + \varepsilon(r) = p_c + \alpha(r^w),
\end{equation}
still with $r=r(z)=\|z\|$. It is noted that this gives $\xi(p_c + \varepsilon(r))  = r^w$ which will be the starting point of our analyses.  We will denote the corresponding measure by 
$\tilde{\mathbb P}_w$ and expectations therein by $\tilde{\mathbb E}_w$.

\begin{remark} It is noted that the formulation in Eq.(\ref{p(r)}) has the slight advantage over the informal description featured in the introduction that it is well-defined at \emph{all} points of the lattice.  Moreover, in cases such as the triangular site percolation model where a logarithmic form of scaling can be established, i.e.
\begin{equation}
\label{nu}
\xi(p) \approx |p-p_c|^{-\nu} \quad (p\downarrow p_c)
\end{equation}
we make direct contact with the more informal description. Indeed using Eqs.(\ref{nu}), (\ref{p(r)}) and (\ref{r}) we get
\begin{equation}
\nu = \lim_{r\to\infty} \frac{\log(\xi(p_c + \alpha(r^w))}{|\log\alpha(r^w)|}
\frac{\log r}{\log r } = w\lim_{r\to\infty}\frac{\log r}{|\log \varepsilon(r)| }
\end{equation}
i.e.~$\log \varepsilon(r)/ \log r \to -w/\nu = -\lambda$, that is to say $\varepsilon(r) \approx r^{-\lambda}$.
\end{remark}

We will now consider the inhomogeneous model as described, and we will denote by $\Psi_N$ the number of sites in $S_N$ that belong to the infinite cluster, and by $\Phi_N$ the number of sites in $S_N$ that are connected to the origin by a path lying entirely in $S_N$. We are ready for the statement of our main theorem:

\begin{theorem}
\label{main}
Consider the quantity
\begin{equation}
\label{I}
I_N := \sum_{z \in S_N} P_{\infty}\big(p_c + \varepsilon(r(z))\big).
\end{equation}
Then 
\begin{enumerate}[(i)]
\item We have $I_N \asymp N^2 \pi(N^w)$, and this quantity measures the size of $\Psi_N$ and $\Phi_N$: As $N \to \infty$,
\begin{equation}
\tilde{\mathbb E}_w(\Psi_N), \:\: \tilde{\mathbb E}_w(\Phi_N) \asymp I_N.
\end{equation}

\item Furthermore, we have the variance estimate: for any $\epsilon > 0$,
\begin{equation}
\tilde{\mathbb V}_w(\Psi_N) \leq C_{2\epsilon} N^{2\epsilon}N^{2+2w} \pi^2(N^w),
\end{equation}
so that $\tilde{\mathbb V}_w(\Psi_N) = o(I_N^2)$ and
\begin{equation}
\frac{\Psi_N}{\tilde{\mathbb E}_w(\Psi_N)} \longrightarrow 1 \quad \text{in $L^2$}.
\end{equation}
\end{enumerate}
Finally, conditionally on $\{0 \leadsto \infty\}$, these results hold for $\Phi_N$ as well.

\end{theorem}
\begin{remark}
Under the assumption of scaling, if we write
\begin{equation}
P_\infty(p) \approx (p-p_c)^{\beta}
\end{equation}
then
\begin{equation}
I_N  \asymp \int_{S_N}d^2r\frac{1}{r^{\lambda \beta}} 
\asymp N^{2-\lambda \beta}.
\end{equation}
A result along these lines can be stated for the triangular site model.
\end{remark}
\begin{corollary}
For the triangular site model (or any model where logarithmic scaling can be established), when $N \to \infty$,
\begin{equation}
\tilde{\mathbb E}_w(\Psi_N) \approx N^{2 - \lambda \beta}.
\end{equation}
\end{corollary}

In the last section we will prove that if $\varepsilon (r) \approx \alpha(r/[\kappa \log\log r])$
there is a $\kappa_c$ above which there is percolation and below which there is not.  We will defer to Section 4 a precise statement of this result.

\section{Proofs}
The following, our key lemma, is an adaptation of the typical sorts of derivations to be found in \cite{CCF}, \cite{Ke3} and \cite{Ke4}.  
\begin{lemma}
Let $\ell(r)$ be standing notation for $L(p_c + \varepsilon(r))$ and $S_{\ell}(z) = S_{\ell(r(z))}(z) = S_{\ell(\|z\|)}(z)$.  Then for any $z$, 
\begin{equation}
\tilde{\mathbb P}_w(\{z \leadsto \infty\}) 
\geq c_1 \tilde{\mathbb P}_w(\{z \leadsto \partial S_{\ell}(z)\}).
\end{equation}
Similarly, if $r(z) < N - \ell(r(z))$ then
\begin{equation}
\label{cPhi}
\tilde{\mathbb P}_w(\{z~ \mathop{_{\leadsto}^{\thinspace S_N}}~ 0\})
\geq c_2 \tilde{\mathbb P}_w(\{z \leadsto \partial S_{\ell}(z)\}).
\end{equation}
In the above, $c_1$ and $c_2$ are constants of order unity independent of 
$z$.  
\end{lemma}
\begin{remark}
Since the above are supplemented with the obvious complementary bounds, the event 
$\{z \leadsto \partial S_{\ell}(z)\}$
is, essentially, necessary and sufficient for $z$ to join the relevant 
large scale IIIC.  This is the sort of result that Kesten established in the uniform system and, in fact, analogous statements are anticipated for all low--dimensional critical systems. Note also that
\begin{equation}
\label{ell}
\ell(r) = L(p_c + \varepsilon(r)) \asymp \xi(p_c + \varepsilon(r)) = r^w \: (\ll r).
\end{equation}
\end{remark}
\begin{proof}
We will establish the above for all $r$ sufficiently large but it is remarked that just how large is sufficient may depend on $w$. 
Let us start with the first case; here, for various reasons, it is worthwhile to know that the connection to infinity can be achieved by moving {\it outward} from the immediate vicinity of the point  $z$.   Consider the event $\mathcal A_{\ell}(z)$ that an occupied ring separates 
$\partial S_{\ell}(z)$ from $S_{\frac 13 \ell(r(z))}(z)$.  Once $\mathcal A_{\ell}(z)$ has occurred, 
with a few more parallelogram crossings, the separating circuit can be attached to a crossing of a $3\ell(r) \times \ell(r)$ parallelogram that is heading, more or less, in a direction away from the origin.  We further intersect this with a few more crossings on a few more scales -- each scale 3 times the previous one.  The number of times we must do this, which is on the order of just a few and not dependent on $r$ will be made precise momentarily; 
the relevant crossings are depicted in Figure 1.  
\begin{figure}[t]
\vspace{-3mm}
\begin{center}
\includegraphics[width=0.8 \textwidth]{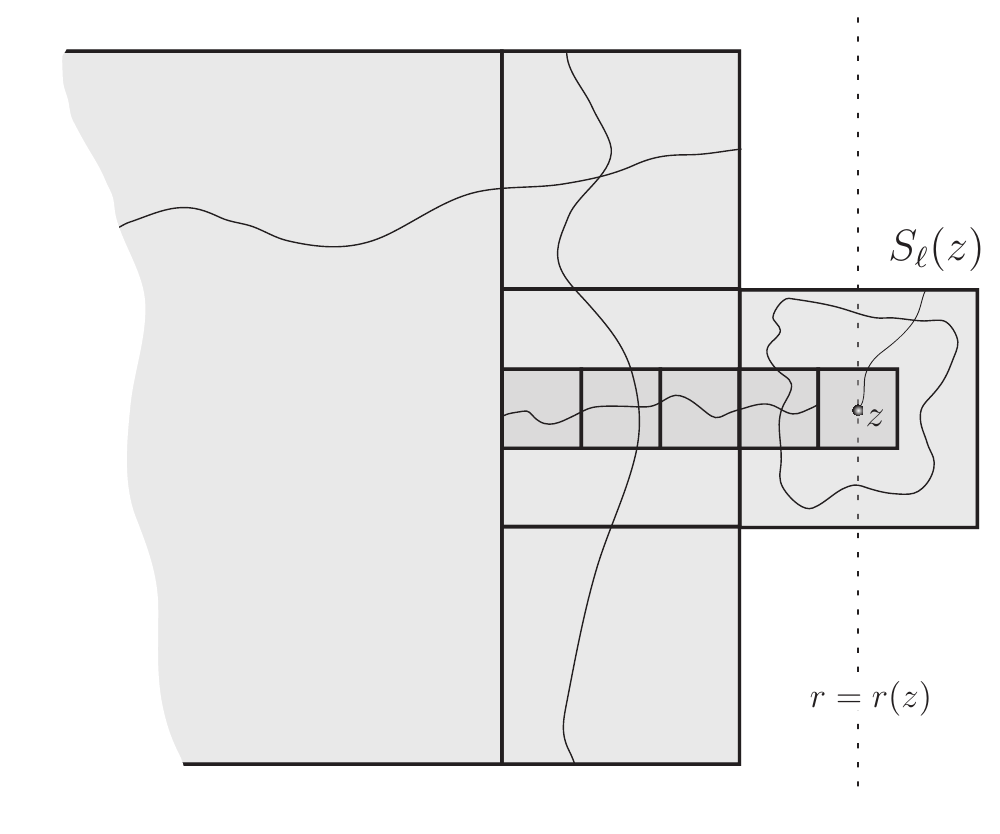}
\vspace{-5mm}
\end{center}
\caption{\footnotesize{
The event $\mathcal A_{\ell}(z)$ and a few subsidiary crossings which serve to attach $z$ to a cluster with diameter moderately larger than the local correlation length.  This cluster is, in turn, easily attached to points twice as far from the origin as $z$ and, ultimately, to infinity.
}}
\vspace{-3mm}
\end{figure}
Denoting the intersection of the annular event and the crossing events alluded to 
by $\mathcal B_\ell$ we have, by FKG,
\begin{equation}
\label{cost}
\tilde{\mathbb P}_w(\{ z\leadsto \partial S_{\ell}
(z) \}\cap \mathcal B_\ell)
\geq 
\tilde{\mathbb P}_w(\{ z\leadsto \partial S_{\ell}
(z) \})\tilde{\mathbb P}_w(\mathcal B_\ell) 
\geq B\tilde{\mathbb P}_w(\{ z\leadsto \partial S_{\ell}
(z) \})
\end{equation}
where $B$ is the probability of $\mathcal B_\ell$ at $p = p_c$.  We remind the reader that this is a uniformly positive constant (obtained by ``Russo--Seymour--Welsh theory'' and a few more applications of the FKG inequality) that does not depend on the particular scale where the action is taking place.  

Now, consider the situation at a distance $2r$ from the origin. Here, by Eq.(\ref{p(r)}) (the definition of $p(z)$), the local correlation length has grown to $2^w$ its size at the distance $r$.  Let us estimate the finite--size correlation length.  First we let $c_3$ and $c_4$ denote the constants by which the two correlation lengths may be compared:
\begin{equation}
c_3 L(p)  \leq   \xi(p)  \leq   c_4 L(p).
\end{equation}
Then, it is seen that $L(p_c + \varepsilon(2r)) \leq 2^{w}c_4c_3^{-1}L(p_c + \varepsilon(r)) = 2^{w}c_4c_3^{-1} \ell(r)$, and it is clear that everywhere in the annular region $S_{2r}(0)\setminus S_{r}(0)$,
the effective finite--size scaling correlation length is going to be uniformly smaller.  The constant $2^{w}c_4c_3^{-1}$ determines the scale of our initial cluster (which, we recall, is attached to the annular ring which in turn is connected to $z$).  Having achieved this scale we are beyond the correlation length as defined by the distance $2r$.  Using $p(2r)$ as a bound for the density in the annular region, it is not of much cost to connect this cluster out to $\partial S_{2r}(0)$.   This may be done, e.g.~by a standard 
``rectangle rescaling program'' -- constructing overlapping crossings the $k^{\text{th}}$ of which has probability in excess of 
$1-c\delta^{3^k}$ and whose scale is $3^k$ times that of the original aggregation. Note however, that we have to have taken $r$ large enough so that $2^{w}c_4c_3^{-1} \ell(r) \leq r$.

We have thus hooked the point $z$ to a cluster that connects 
$\partial S_{r}(0)$ to $\partial S_{2r}(0)$ at an additional probabilistic cost, beyond what is in Eq.(\ref{cost}), of no more than $\prod_k(1-c\delta^{3^k}) > 0$ -- again using repeatedly the FKG inequality.  The scale $r$ cluster can now be directed to infinity by straightforward arguments (of a similar nature) which may be directly taken from \cite{CCD} Theorem 2.  

The second bound, Eq.(\ref{cPhi}), is proved in a similar fashion -- actually easier because, \ae sthetics aside, we are {\it forced} to work inwards.  The first few steps are identical: assuming that $\{z \leadsto \partial S_\ell(z)\}$ has occurred, 
we use the event $\mathcal A_\ell(z)$ and some more crossings to hook $z$ up to a $3\ell(r)\times \ell(r)$ crossing -- this time headed in the general direction of coordinate decrease. But now, agreeing to always head inwards, we may do a $\times 3$ rescaling program without apology since $p(z)$ is only getting bigger.  Thus, we continue till we reach the boundary of $S_{r/2}(0)$, again at a cost of no more than $\prod_k(1-c\delta^{3^k}) > 0$.  With probability that is (stretched exponentially) close to one there is an occupied ring in 
$S_{r}(0)\setminus S_{r/2}(0)$; this may be obtained by summing Eq. (\ref{apb}) over both boundaries.   Finally, with non--zero probability, the event $\{0 \leadsto \partial S_r(0)\}$ occurs and it is clear that the intersection of all these events produces the event 
$\{z \negthinspace \mathop{_{\ ^{\leadsto}\ }^{\thinspace S_N}} \negthinspace 0\}$.
 As before, we have made repeated use of FKG and it is noted that
the probabilities of all the relevant events save
$\{z\leadsto \partial S_{\ell}(z)\}$
are of order unity independent of $z$.
\end{proof}

\noindent {\it Proof of Theorem \ref{main}}.~
The above lemma proves almost completely the portion of Theorem \ref{main} which concerns expectations.  Indeed, for a uniform system,
\begin{equation}
P_{\infty}(p)  \asymp \mathbb \pi(L(p))
\end{equation}
by Eq.(\ref{KClassic}), so the summand in Eq.(\ref{I}) may be replaced by $\pi(r^w(z))$ (as $\ell(r) \asymp r^w$ by Eq.(\ref{ell})), and it is easily seen that $\pi(r^w(z)) \asymp \tilde{\mathbb P}_w(\{z\leadsto \partial S_{\ell}(z)\})$:  indeed, as $z^\prime$ varies throughout $S_\ell(z)$, the local correlation length varies by a fractional amount which is only of the order $[r(z)]^{-(1-w)}$.  So, we may as well estimate by the largest value of $p$ within $S_\ell(z)$ and use the associated slightly smaller $L$.  But then, using bounds as in Eq.(\ref{E}) and Eq.(\ref{KClassic}), we get that 
$I_N \asymp \sum_{z\in S_N}\tilde{\mathbb P}_w(\{z \leadsto \partial S_{\ell}(z)\})$, which is our asymptotic expression for $\tilde{\mathbb E}_w(\Psi_N)$.

Before we dispense with $\tilde{\mathbb E}_w(\Phi_N)$ let us first verify the (asymptotic) evaluation of the quantity $I_N$.  We already have that 

\begin{equation}
 \label{DOVE}
I_N \asymp \sum_{z \in S_N}
 \pi(r^w(z)).
\end{equation}
Let us take a logarithmic division of $S_N$:  define $k=k(N)$ so that $2^k < N \leq 2^{k+1}$, then 
\begin{equation}
 \label{DIV}
I_N \asymp \sum_{j\leq k} (2^j)^2\pi(2^{jw}) + \mathcal E(k)
\end{equation}
where $\mathcal E(k)$ is no more than the order of $N^2\pi(N^w)$. In the above, we have used Eq.(\ref{E}) on more than one occasion.  Obviously, the purported principal term is at least of this order so there is no further need to consider $\mathcal E$.  We pull out the leading term in the sum:
\begin{equation}
\sum_{j\leq k} (2^j)^2\pi(2^{jw}) \asymp 2^{2k}\pi(2^{kw})
\sum_{j \leq k}2^{2(j-k)}\frac{\pi(2^{jw})}{\pi(2^{kw})}
\end{equation}
Now we use the fact that $\pi(2^{wk})/\pi(2^{wj}) \asymp \pi(2^{wj}|2^{wk})$ (using Eq.(\ref{F})) so that the coefficient of $2^{2k}\pi(2^{kw})$ (which is also at least as {\it large} as the order of unity because of the last term in the sum) is no more than
\begin{equation}
\label{See}
\tilde{c}_5  
=  \sum_{q = 0}^\infty
2^{-q(2-w\mu)} < \infty
\end{equation}
since $\mu$ is certainly less than two.
It is obvious given Eq.(\ref{DOVE}) for $I_N$ that $\tilde {\mathbb E}(\Phi_N)$ is (asymptotically) bounded above by $I_N$ and below by $I_{\frac{N}{2}}$ which 
by now, are seen to be comparable to each other.

Let us turn now to the variance bound. We first note that we can write $\Psi_N = \sum_{x \in S_N} \mathbb{I}_{\{x \leadsto \infty\}}$, so that
\begin{align*}
\tilde{\mathbb V}_w(\Psi_N) & = \sum_{x,y \in S_N} \Bigg[\tilde{\mathbb P}_w(\{x \leadsto \infty\}, \{y \leadsto \infty\}) - \tilde{\mathbb P}_w(\{x \leadsto \infty\}) \tilde{\mathbb P}_w(\{y \leadsto \infty\})\Bigg]\\
& = \sum_{x,y \in S_N} \Bigg[\tilde{\mathbb P}_w(F_x \cap F_y) - \tilde{\mathbb P}_w(F_x) \tilde{\mathbb P}_w(F_y)\Bigg]
\end{align*}
where we have used the notation $F_x = \{x \leadsto \infty \}$. Now recall that $\ell(r) \asymp r^w$ (Eq.~(\ref{ell})): as $w<1$, we can find some $\epsilon>0$ such that $w+\epsilon<1$. We introduce the enhanced length $l(r) = \ell(r)r^\epsilon$   (which is still $\ll r$) and 
as above, we abbreviate $S_{l(\|x\|)}(x)$ by $S_{l}(x)$.  We denote by 
$F'_x$ the event $\{x \leadsto \partial S_{l}(x) \}$. 
It is not hard to check that there is a $b>0$ (independent of $x$) such that for $\|x\|$ sufficiently large,
\begin{equation}
\label{NH}
\tilde{\mathbb P}_w(F_x \Delta F'_x) \leq e^{-\|x\|^b}.
\end{equation}
We deduce, for $n$ some small power of $N$, that 
\begin{align*}
\tilde{\mathbb V}_w(\Psi_N) \leq 
  17C_0^2 n^4~
+&~~C_1 N^4 e^{-n^{b}}~ + ~2 
\negthickspace  \negthickspace
\negthickspace  \negthickspace
 \sum_{x\in S_{n}, y \in S_N \setminus S_{3n}}\Bigg[ \tilde{\mathbb P}_w(F_x \cap F'_y) - \tilde{\mathbb P}_w(F_x) \tilde{\mathbb P}_w(F'_y) \Bigg] \\
& + \sum_{x,y \in S_N \setminus S_{n}}\Bigg[ \tilde{\mathbb P}_w(F'_x \cap F'_y) - \tilde{\mathbb P}_w(F'_x) \tilde{\mathbb P}_w(F'_y) \Bigg].
\end{align*}
The first term serves to estimate the terms in which $\{x\in S_n, y \in S_{3n}~ + \longleftrightarrow \}$ where, as we recall, $C_0$ is the constant that figures into the volume of a box, and the reader is invited to verify the factor of 17.  Whenever $x$ is in $S_{n}^c$, we replace $F_x$ with $F'_x$ and similarly for $y$; the error incurred is accounted for in the second term (and we have assumed that $n$ is large enough so that the bound in Eq.(\ref{NH}) is safely in effect).  The last two terms are self--explanatory and will be dispensed with below.  

Let us start with the first sum.  For $y\in S_{3n}^c$ and $x \in S_{n}$, it is observed that, for $n$ large enough, $S_l(y)$ is disjoint from $S_n$.  Suppose that an occupied circuit surrounding $S_l(y)$ separates it from $S_n$.  Now the event $F'_y$ depends only on the configuration inside 
$S_{l}(y)$ while (conditioning on the innermost such ring) the event $F_x$ depends only on the configuration outside and, perhaps, including the ring.  I.e.~given such a ring, the events $F_x$ and $F'_y$ are conditionally independent.  The probability of $F'_y$ is unchanged while 
the probability of $F_x$ {\it and} the ring event is bounded above by $\tilde{\mathbb P}_w(F_x)$ alone.
Thus we learn for 
$y\in S_{3n}^c$ and $x \in S_{n}$
that 
\begin{align}
\label{NR}
\tilde{\mathbb P}_w(F_x\cap F'_y) &-
\tilde{\mathbb P}_w(F_x)\tilde{\mathbb P}_w(F'_y)
\notag  \\ 
& \leq \tilde{\mathbb P}_w(\{ \text{no occupied circuit separates } S_n \text{ from } S_l(y) \}).
\end{align}
The right--hand side of Eq.(\ref{NR}) is bounded by another term of the order 
$e^{-n^{b}}$ and we may thus absorb the entire first sum into the second error term at the expense of shifting the index of the constant.  

We turn attention to the final term in the above written bound on the variance.
If $x$ and $y$ are distant enough, $S_{l}(x)$ and $S_{l}(y)$ are disjoint, and the events $F'_x$ and $F'_y$ are independent. Now note that $l(\|x\|) \geq l(\|y\|)$ if $\|x\| \geq \|y\|$, so that $S_{l}(x) \cap S_{l}(y) = \varnothing$ if $\|x\| \geq \|y\|$ and $y \in S_{3l}^c(x)$. Hence,
\begin{align*}
\sum_{x,y \in S_N \setminus S_{n}} \Bigg[\tilde{\mathbb P}_w(F'_x \cap F'_y) - \tilde{\mathbb P}_w(F'_x) \tilde{\mathbb P}_w(F'_y)\Bigg] & \\
& \hspace{-3cm} \leq 2  
\negthickspace
\negthickspace
\negthickspace
\negthickspace
 \sum_{x,y \in S_N \setminus S_{n}, \|x\| \geq \|y\|} \Bigg[\tilde{\mathbb P}_w(F'_x \cap F'_y) - \tilde{\mathbb P}_w(F'_x) \tilde{\mathbb P}_w(F'_y)\Bigg]\\
& \hspace{-3cm} \leq 2 \sum_{x \in S_N \setminus S_{n}} \sum_{y \in S_{3l}(x) \cap S_n^c} \Bigg[\tilde{\mathbb P}_w(F'_x \cap F'_y) - \tilde{\mathbb P}_w(F'_x) \tilde{\mathbb P}_w(F'_y)\Bigg]
\end{align*}

We now have to estimate, for a site $x \in S_N \setminus S_{n}$, the sum
\begin{equation}
\Sigma_w(x) = \sum_{y \in S_{3l}(x) \cap S_n^c} \tilde{\mathbb P}_w(x \leadsto \partial S_{l}(x), y \leadsto \partial S_{l}(y))
\end{equation}
Note that for $y$ inside $S_{3l}(x) \cap S_n^c$, the size $l(\|y\|)$ of the associated box does not vary too much and certainly (since $\|y\|$ is already larger than $n$) always satisfies $l(\|y\|) \geq \ell(\|x\|)$.  
Thus we have
\begin{equation}
\Sigma_w(x) \leq \sum_{y \in S_{3l}(x)} \tilde{\mathbb P}_w(x \leadsto \partial S_{\ell(\|x\|)}(x), y \leadsto \partial S_{\ell(\|x\|)}(y))
\end{equation}
We can then proceed by summing over concentric annuli centered on $x$ cutting down even further on what we require in accord with $\|x-y\|$: take $k = k(x)$ such that $2^k < \ell(\|x\|) \leq 2^{k+1}$. If $y$ is outside of $S_{2^{k+1}}(x)$, the two boxes $S_{2^k}(x)$ and $S_{2^k}(y)$ are disjoint.  Hence, for these cases,
\begin{align}
\tilde{\mathbb P}_w(x \leadsto \partial S_{\ell(\|x\|)}(x),& \thinspace y \leadsto \partial S_{\ell(\|x\|)}(y)) 
\notag  \\
\label{far}
& \leq \tilde{\mathbb P}_w(x \leadsto \partial S_{2^k}(x)) \tilde{\mathbb P}_w(y \leadsto \partial S_{2^k}(y)) \asymp \pi^2(2^k)
\end{align}
and the number of such terms does not exceed the volume of
$S_{3l}(x)$.  Thus, the total contribution from these well--separated terms is bounded by $C_3l^2(\|x\|)\pi^2(\ell(\|x\|))$ where $C_3$ is a constant not dissimilar from $C_0$.

Now if $y \in S_{2^{j+1}}(x) \setminus S_{2^j}(x)$ with $j \leq k-3$, we have by independence
\begin{align*}
\tilde{\mathbb P}_w(x \leadsto \partial S_{\ell(\|x\|)}(x), y \leadsto \partial S_{\ell(\|x\|)}(y)) & \\
& \hspace{-3cm} \leq \tilde{\mathbb P}_w(x \leadsto \partial S_{2^{j-1}}(x)) \tilde{\mathbb P}_w(y \leadsto \partial S_{2^{j-1}}(y)) \tilde{\mathbb P}_w(\partial S_{2^{j+2}}(x) \leadsto \partial S_{2^k}(x))\\
& \hspace{-3cm} \leq C_4 \pi^2(2^{j-1}) \pi(2^{j+2} | 2^k)\\
& \hspace{-3cm} \leq C_5 \pi(2^j)\pi(2^k).
\end{align*}
using once again Eqs.(\ref{E}) and (\ref{F}). If $j \geq k-2$, we just drop the last term $\tilde{\mathbb P}_w(\partial S_{2^{j+2}}(x) \leadsto \partial S_{2^k}(x))$ in the first inequality: since in this case $\pi(2^j) \asymp \pi(2^k)$, the final inequality still holds. Hence, we must sum $\sum_{j \leq k}(2^{j})^2\pi(2^j)\pi(2^k)$.  This is identical to the previous argument:  pulling out an overall factor of $[2^k\pi(2^k)]^2$, the resulting summand may be expressed as $[\pi(2^j\mid2^k)2^{2(k-j)}]^{-1}$, and if we use the bound in 
Eq.(\ref{alpha}) with $\mu < 2$, we see
\begin{equation}
\sum_{j\leq k}2^{2j}\pi(2^j)\pi(2^k) \leq C_6 [2^k\pi(2^k)]^2.
\end{equation}
This is somewhat smaller than the contribution from the well--separated terms (Eq.(\ref{far})) so, overall,  
\begin{equation}
\Sigma_w(x)  \leq   C_7 l^2(\|x\|)\pi^2(\ell(\|x\|)).
\end{equation}

We finally sum on $x$ to conclude
\begin{align*}
\tilde{\mathbb V}_w(\Psi_N) & \leq 
17C_0^2 n^4 + C_2 N^4 e^{-n^{b }} + 2 \negthickspace  \sum_{x \in S_N \setminus S_{n}}
 \negthickspace  \Sigma_w(x) \\
 &  \leq 17C_0^2 n^4 + C_2 N^4 e^{-n^{b }}
 +C_7 N^{2\epsilon}\sum_{j = 1}^{k(N)}2^{2j}2^{2wj}\pi^2(2^{jw}).
\end{align*}
In the above, all indexed constants are numbers which are uniformly of order unity.
As in previous arguments, we may bound the sum by a constant times 
$N^{2 + 2w}\pi^2(N^w)$ and, finally, we choose $n$ a small enough power of $N$ so that $n^4$ is relatively negligible -- which will still easily diminishes the other ``error term''.
Recalling that $I_N^2 \approx N^4 \pi^2(N^w)$ -- and that $w + \epsilon < 1$ -- we have obtained the desired statement about $\tilde{\mathbb V}_w (\Psi_N)$.

Concerning $\Phi_N$, although we have to be a bit more cautious, the proof remains essentially the same. Here we can write
\begin{align*}
 \tilde{\mathbb V}_w(\Phi_N \mid\{0 \leadsto \infty\}) 
 = ~~~~~~~~~~~~~~~~~~~~~~~~~~~~~~~~
 ~~~~~~~~~~~~~~~~~~~~~~~~~~~~~~~~~~~~~~~~~~~~~~~~~~~ \\
 \sum_{x,y \in S_N} \Bigg[ \tilde{\mathbb P}_w(\{x \leadsto 0, y \leadsto 0 \}|\{ 0 \leadsto \infty\}) - \tilde{\mathbb P}_w(\{x \leadsto 0\} |\{ 0 \leadsto \infty\}) \tilde{\mathbb P}_w(\{y \leadsto 0 \}| \{0 \leadsto \infty\}) \Bigg]\\
 = \sum_{x,y \in S_N} \Bigg[ \frac{\tilde{\mathbb P}_w(F_x \cap F_y \cap \{0 \leadsto \infty\})}{\tilde{\mathbb P}_w(\{0 \leadsto \infty\})} - \frac{\tilde{\mathbb P}_w(F_x \cap \{0 \leadsto \infty\}) \tilde{\mathbb P}_w(F_y \cap \{0 \leadsto \infty\})}{\tilde{\mathbb P}_w(\{0 \leadsto \infty\})^2} \Bigg]
\end{align*}
We again cut out a central portion at the cost of the order $n^4$ and we are left with two principal contributors the first of which is given by (twice) the sum with $x\in S_n$ and $y\in S_{3n}^c$.  Here, using another argument involving a separating ring, the positive term is bounded as follows:
\begin{equation}
\tilde{\mathbb P}_w(F_x \cap F_y\mid\{0 \leadsto \infty\}) \leq 
\tilde{\mathbb P}_w(F_x \mid\{0 \leadsto \infty\})
\tilde{\mathbb P}_w(F'_y) + \mathcal N_R
\end{equation}
where $\mathcal N_R$ is the ``no ring'' event described in Eq.(\ref{NR}).  Meanwhile, 
$$
\tilde{\mathbb P}_w(F_y \mid\{0 \leadsto \infty\})
\geq  \tilde{\mathbb P}_w(F_y)
$$
so we are left with 
$\tilde{\mathbb P}_w(F_x \mid\{0 \leadsto \infty\}) \tilde{\mathbb P}_w(F'_y\Delta F_y )$ plus the $\mathcal N_R$ term both of which are of the order $e^{-n^{b}}$.  

We are left with the principal term and first off (at small cost) we replace the events $F_x$ and $F_y$ by the events $F'_x$ and $F'_y$.  Here in addition we will replace $\{0 \leadsto \infty\}$
by the event
\begin{equation}
F^{(x,y)}_0  =   \{0 \leadsto \infty \text{ outside of $S_{l}(x) \cup S_{l}(y)$} \}
\ \ (= \{ 0 \mathop
{_{\ ^{~~~~-\negthinspace-\negthinspace-\negthinspace\leadsto}\ }
^{\ [S_{l}(x) \cup S_{l}(y)]^c}} \negthinspace \infty \}).
\end{equation}
It is not hard to see that the two events are very close.  Indeed while ostensibly 
$F_0^{(x,y)} \supset \{0\leadsto \infty \}$, in the event that $S_{l}(x)$ and 
$S_l(y)$ are both surrounded by occupied circuits which separate these boxes from the origin, the conditional probability is larger.  But since we are well away from the origin these sorts of separating rings occur with probability close to one and we get an upper bound similar to that of Eq.(\ref{NH}) for 
$\tilde{\mathbb P}_w(F_0^{(x,y)} \Delta  \{0\leadsto \infty \})$. 
The remainder  of the proof is essentially identical.
\qed

\section{A Sharp Transition}

To treat the marginal case, we take
\begin{equation}
\label{loglog}
p(\negthinspace z) := p_c + \varepsilon(r) = p_c + \alpha( r /\kappa\log \log(r))
\end{equation}
with $\kappa$ a constant and it is assumed that $r$ is large enough so that all quantities are positive (otherwise, we set $p = 1$).  We denote by $\tilde{\mathbb P}_{1,\kappa}$ the associated inhomogeneous probability measure. Note that this gives 
$\xi(p_c + \varepsilon(r))  = r / \kappa\log\log(r)$ which is hardly distinguishable from linear in $r$.  Nevertheless, we will prove

\begin{theorem} 

\label{final}
For the 2D inhomogeneous percolation  models defined via Eq.(\ref{loglog}),
there is a critical value $\kappa_c \in (0,+\infty)$ such that for 
$\kappa > \kappa_c$ there exists $\tilde{\mathbb P}_{1,\kappa}$-- a.s.~an infinite cluster, while for $\kappa < \kappa_c$ there is $\tilde{\mathbb P}_{1,\kappa}$-- a.s.~no infinite cluster.
\end{theorem}

\begin{proof}
By monotonicity, it suffices to prove that there exists a value of $\kappa$ for which the system percolates, and another value for which it does not percolate.  We start with the percolative part.  

Consider the crossing of any $3r \times r$ parallelogram that is situated so that the maximum distance form the origin is no more than $Mr$ with $M$ a (uniform) constant of order unity.  Within this parallelogram, the lowest value of $p$ estimates a uniform value for the density.  This in turn provides a finite--size scaling correlation length which is smaller than $r/[q_1\log\log r]$ for some constant, $q_1$, which is large if $\kappa$ is large.  By starting at {\it this} length scale, and instituting an $\times 3$ rescaling program till the scale of the $3r \times r$ is reached, it is seen that the probability of a crossing at the larger scale is at least $1 - Q_1\delta^{q_1\log\log r}$.  Here $Q_1$ is a constant of order unity -- perhaps small -- but independent of $r$ and $\kappa$.  Writing $\delta$ as an exponential this bounds the probability of crossing the parallelogram at scale $r$ below by 
$1 - Q_2/(\log r)^{q_2}$ where $Q_2$ is of order unity independent of $r$ and $\kappa$ and $q_2$ is large if $\kappa$ is large.    

We now consider a sequence of overlapping $3\times 1$ parallelograms at a sequence of scales with each scale thrice the previous one.  Here the sequence is such that the smallest scale is in the vicinity of the origin and the event of simultaneous crossings of all of them (or all but a finite number of them) implies the existence of an infinite cluster.  \negthinspace
$\negthinspace \footnote{~E.g. in the ``T'' construction in  \cite{CCD}, which the reader may wish to check, there are two rectangles at each scale; although one of them was $4\times 1$ this was only for \ae sthetic reasons and, in any case, the above mentioned bounds on crossing probabilities  are easily extended to parallelograms with any finite aspect ratio.}$
~If the scale of the $k^{\text {th}}$ rectangle is simply a constant times $3^k$, the $\tilde{\mathbb P}_{1,\kappa}$ probability of seeing all the crossings is bounded below by 
$$
g(\kappa) =
\prod _k \bigg[1 - \frac{Q_3}{k^{q_2}}\bigg]
$$
(with $Q_3$ another uniform constant) which is positive for all $\kappa$ large enough.  The quantity $g(\kappa)$ bounds the probability that the origin belongs to an infinite cluster, the a.s.~existence of an infinite cluster follows from an application of the Borel--Cantelli lemma.  It is remarked that by the consideration of large scale circuits -- which are present even at $p = p_c$ -- the infinite cluster is a.s.~unique.

\begin{figure}[t]
\vspace{-3mm}
\begin{center}
\includegraphics[width=0.8 \textwidth]{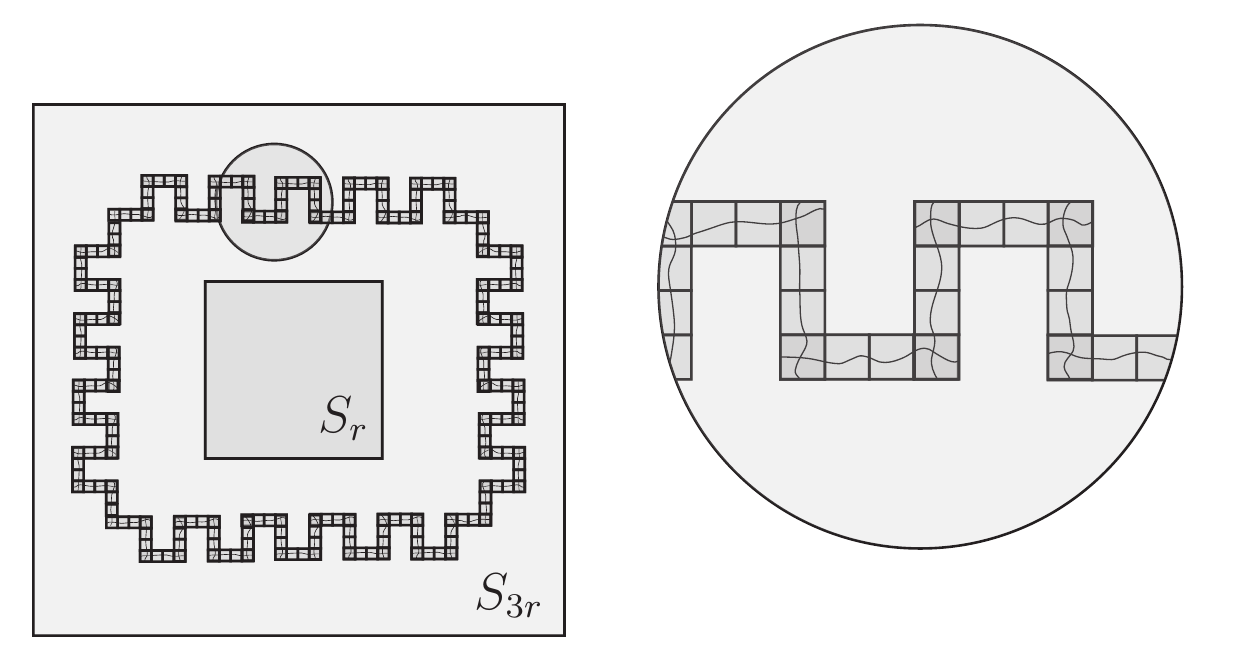}
\vspace{-5mm}
\end{center}
\caption{\footnotesize{
Construction of a dual circuit in the annulus
$S_{3r} \setminus S_r$ for the borderline case.  If $\kappa$ is small, the probability of these circuits tends to zero with a small power of $\log r$ and  
percolation is prevented.
}}
\vspace{-3mm}
\end{figure}

For the non--percolative result, when $\kappa$ is small, we shall consider events in the annular regions $S_{3r} \setminus S_r$.  Within this region, $L(p(z))$ is now uniformly larger than $r/[a_1 \log\log r]$ where $a_1$ is small if $\kappa$ is small.  This implies that the long--way crossings of 
$4\times 1$ parallelograms occur with probability of order unity.  These crossings may be stitched together, e.g.~in a square--wave fashion, to construct a dual circuit in the annulus; see Figure 2.  Using FKG, the probability of such a ring can be bounded below by $A_1/(\log r)^{a_2}$, where $a_2$ is small if $\kappa$ is small.  Once more looking at interlocking annuli at scales $\propto 3^k$, this translates into a probability $\propto k^{-a_3}$ where $a_3$ is small if $\kappa$ is small.  Divergence of 
$\sum_k k^{-a_3}$ implies the a.s.~presence of $\infty$-ly many of these dual circuits and, therefore, $\tilde{\mathbb P}_{1,\kappa}$--a.s.~no infinite cluster.
\end{proof}

\section{Appendix}

Here we provide the promised derivation that, in context of 2D percolation models of the sort described in Section 2, all correlation lengths are asymptotically equivalent.  As the reader will note, the key is already in the proof of Theorem \ref{final}.  

\medskip

\noindent \textit{Proof that} $L(p) \asymp \xi(p)$.~~~  Let us start by defining 
$R_{3,N}(p)$ to be the probability of a long--way crossing of a $3N\times N$ parallelogram and $D_{3,N}  = 1 - R_{3,N}$ the probability of the complimentary event, namely a short--way dual--crossing of this shape.  
First, it is claimed that
\begin{equation}
\lim_{N\to\infty}D_{3,N}^{\frac{1}{N}} = e^{-\frac{1}{\xi}}.
\end{equation}
Indeed, $D_{3,N} \leq  V_{1}N^{2}e^{-\frac{N}{\xi}}$ by the \textit{a priori} bounds discussed in Eq.(\ref{apb}) where $V_1$ is a uniform constant (equal to 9 on the square lattice).  On the other hand, we may obtain a lower bound for $D_{3,N}$ by just allowing the site at the center of the base to connect to its counterpart across the way.  While, ostensibly, this would allow for paths to ``leak out the ends'', it is not hard to show that the probability of such a huge lateral excursion is as small as $e^{-\frac{3}{2}\frac{N}{\xi}}$ so, for all intents and purposes, $D_{3,N} \gtrsim \tau^{\ast}_N$ which establishes the limit.  
Using the $3\times$ construction discussed at several points earlier in the text and using e.g. $\delta = e^{-1}$, we get that
\begin{equation}
R_{3,3^{k}L} \geq 1 - ce^{-3^{k}}.
\end{equation}
Thus, for some sequence of $N$'s, $D_{3,N} \leq ce^{-\frac{N}{L}}$ which implies $e^{-1/\xi} \leq  e^{-1/L}$.  
Now consider the probability of a hard--way dual crossing of a $4\times 1$ parallelogram of scale $L^{\prime}$ which is less than $L$ but, say, larger than $\frac{1}{2}L$.  This occurs with a probability of order unity independent of $p$ (by Russo-Seymour-Welsh theory) and, as was just done in the last proof, by stitching together the order of $N/L$ such rectangles the desired event is produced.  Thus we have $D_{3,N} \geq  e^{-\sigma\frac{N}{L}}$ for some constant $\sigma$ (which is uniform in $p$) and hence $e^{-1/\xi} \geq  e^{-\sigma/L}$.
\qed

\section*{Acknowledgments}
\noindent 
The authors are grateful to the IPAM institute at UCLA for their hospitality and support during the ``Random Shapes Conference" which was supported by the NSF under the grant~DMS-0439872.
  L.C. was supported by the NSF under the  grant~DMS-0306167.

\end{document}